\documentclass[11pt]{article}%
\usepackage{mathrsfs}
\usepackage{bbm}
\usepackage{amsfonts}
\usepackage{amsmath,amssymb}
\usepackage{amsmath}
\usepackage{amssymb}
\usepackage{graphicx}%
\usepackage{shorttoc}
\usepackage{tikz}
\usetikzlibrary{arrows, automata}
\usetikzlibrary{calc}
\usetikzlibrary{positioning}
\newcommand*{\circled}[1]{\lower.7ex\hbox{\tikz\draw (0pt, 0pt)%
    circle (.5em) node {\makebox[0.4em][c]{\small #1}};}}

\providecommand{\U}[1]{\protect \rule{.1in}{.1in}}

\newtheorem{theorem}{Theorem}[section]
\newtheorem{corollary}[theorem]{Corollary}
\newtheorem{definition}[theorem]{Definition}
\newtheorem{example}[theorem]{Example}
\newtheorem{lemma}[theorem]{Lemma}

\newtheorem{remark}[theorem]{Remark}

\newenvironment{proof}[1][Proof]{\noindent \textbf{#1.} }{\  $\Box$}
\numberwithin{equation}{section}

\begin{document}

\title{\textbf{A Strong Law of Large Numbers under Sublinear Expectations  }}
\author{Yongsheng Song\thanks{RCSDS, Academy of Mathematics and Systems Science, Chinese Academy of Sciences, Beijing 100190, China, and
School of Mathematical Sciences, University of Chinese Academy of Sciences, Beijing 100049, China. E-mail:
yssong@amss.ac.cn.}  }

\date{}
\maketitle

\begin{abstract}  We consider a sequence of i.i.d. random variables $\{\xi_k\}$under a sublinear expectation $\mathbb{E}=\sup_{P\in\Theta}E_P$. We first give a new proof to the fact that, under each $P\in\Theta$, any cluster point of the empirical averages $\bar{\xi}_n=(\xi_1+\cdots+\xi_n)/n$ lies in $[\underline{\mu}, \overline{\mu}]$ with $\underline{\mu}=-\mathbb{E}[-\xi_1], \overline{\mu}=\mathbb{E}[\xi_1]$. Then, we consider sublinear expectations on a Polish space $\Omega$, and show that for each constant $\mu\in [\underline{\mu},\overline{\mu}]$, there exists a probability $P_{\mu}\in\Theta$ such that
\begin {eqnarray}\label {intro-a.s.}
\lim_{n\rightarrow\infty}\bar{\xi}_n=\mu, \ P_{\mu}\textmd{-a.s.},
\end {eqnarray}
 supposing that $\Theta$ is weakly compact and $\{\xi_n\}\in L^1_{\mathbb{E}}(\Omega)$. Under the same conditions, we can get a generalization of (\ref {intro-a.s.}) in the  product space $\Omega=\mathbb{R}^{\mathbb{N}}$ with $\mu\in [\underline{\mu},\overline{\mu}]$ replaced by $\Pi=\pi(\xi_1, \cdots,\xi_d)\in [\underline{\mu},\overline{\mu}]$, where $\pi$ is a Borel measurable function on $\mathbb{R}^d$, $d\in\mathbb{R}$. Finally, we characterize the triviality of the tail $\sigma$-algebra of i.i.d. random variables under a sublinear expectation.
\end{abstract}


\textbf{MSC-classification}: 60F15; 60G50; 60G65

\section{Introduction}
 Let $\{\xi_k\}_{k\ge 1}$ be a sequence of independent and identically distributed random variables  under a sublinear expectation $\mathbb{E}=\sup_{P\in\Theta}E_P$. Set $\bar{\xi}_n=\frac{\xi_1+\cdots+\xi_n}{n}$, $\underline{\mu}=-\mathbb{E}[-\xi_1]$, $\overline{\mu}=\mathbb{E}[\xi_1]$.  Peng, S. (\cite{P08b}) proved that
  \begin {eqnarray}
\lim_{n} \mathbb{E}[\phi(\bar{\xi}_n)]=\sup_{y\in[\underline{\mu}, \overline{\mu}]}\phi(y),
\end {eqnarray} for any $\phi\in C_{b, Lip}(\mathbb{R})$, the collection of bounded and Lipschitz continuous functions on $\mathbb{R}$. This is called the (weak) law of large numbers under sublinear expectations (wLLN*).

Recently, Song (2021) (\cite {Song19}) gave the following error estimate for wLLN*:
\begin {eqnarray}
\sup_{|\phi|_{Lip}\le1}\bigg| \mathbb{E}[\phi(\bar{\xi}_n)]-\sup_{y\in[\underline{\mu}, \overline{\mu}]}\phi(y)\bigg|\le C n^{-\frac{\beta}{1+\beta}},
\end {eqnarray} where $C$ is a constant depending only on $\mathbb{E}[|\xi_1|^{1+\beta}]$.

In this paper, we consider the strong law of large numbers under sublinear expectations (sLLN*). \cite {Chen16} and \cite {CHW19} gave the following bounds for the cluster points  of empirical averages:  For any $P\in\Theta$,
\begin {eqnarray} \label {sLLN-bounds-intro}
\underline{\mu}\le\liminf_{n\rightarrow\infty}\bar{\xi}_n\le\limsup_{n\rightarrow\infty}\bar{\xi}_n\le\overline{\mu}, \ P\emph{-a.s.}
\end {eqnarray}
In this paper, we shall give an alternative proof to the above bounds of the cluster points based on the error estimates for wLLN* given in (\cite {Song19}).

The main purpose of this paper is to study the limit properties of empirical averages of i.i.d. random variables $\{\xi_k\}$ under a sublinear expectation $\mathbb{E}=\sup_{P\in\Theta}E_P$ which is defined on a Polish space $\Omega$.

First, we investigate under which conditions the collection of the cluster points of empirical averages coincides with the interval $[\underline{\mu},\overline{\mu}]$, which is not always the case shown by some counterexamples. Generally, we have the following result (Theorem \ref {LLN-P-kappa}):

\emph{Suppose that $\mathbb{E}=\sup_{P\in\Theta}E_P$ with $\Theta$ weakly compact, and that the random variables $\{\xi_k\}$ are contained in $L^{\gamma}_\mathbb{E}(\Omega)$ for some $\gamma>1$. Then, for any $\mu\in[\underline{\mu},\overline{\mu}]$, we have $P_\mu\in\Theta$
such that
\begin {eqnarray}\label {sLL-intro}
\bar{\xi}_n\rightarrow \mu, \ P_\mu\emph{-a.s.}
\end {eqnarray}
}
We first prove that $\bar{\xi}_n$ converges to $\mu$ in probability (Lemma \ref {lemma-LLN-inP}), which is based on wLLN* given in Peng (2008) (\cite{P08b}), and then prove that $\bar{\xi}_n\rightarrow \mu, \ P_\mu\emph{-a.s.}$ (Theorem \ref {lemma-LLN-as}), which is based on the error estimates for wLLN* given in (\cite {Song19}).

Besides, we characterize the the triviality of the tail $\sigma$-algebra of i.i.d. random variables under a sublinear expectation $\mathbb{E}$ defined on the product space $\Omega=\mathbb{R}^{\mathbb{N}}$ (Section 4).

\emph{For a regular sublinear expectation $\mathbb{E}$ on $\Omega$, let $\Theta$ be the maximal set that represents $\mathbb{E}$, and let $\Theta^e$ be the set of  extreme points of  $\mathbb{E}$.}

$\textbf{1.}$ \emph{Assume the canonical process $\{\xi_k\}$  defined on $\Omega$ is i.i.d under $\mathbb{E}$, i.e., $\xi_{k+1}$ is independent of $\xi_1, \cdots, \xi_k$, for any $k\in\mathbb{N}$.}
\emph{\begin {itemize}
\item [i)] For any Borel function $\pi: \mathbb{R}^d\rightarrow[\underline{\mu}, \overline {\mu}]$, $d\in\mathbb{N}$, we can find a probability $P\in\Theta^e$ such that
\[\lim_{n\rightarrow \infty}\bar{\xi}_n =\pi(\xi_1, \cdots, \xi_d), \ P\emph{-a.s.}\]
So, generally, the tail $\sigma$-algebra is not necessarily trivial under $P\in \Theta$, even for $P\in \Theta^e$. Furthermore, probabilities respecting the above property can be chosen in any compact set $\Theta_0\subset\Theta$ that represents $\mathbb{E}$.
\item [ii)] But it can be shown that the tail $\sigma$-algebra is trivial under each $P\in\Theta^e$ which is  shift invariant, i.e., $P\circ\theta_n^{-1}=P$ for some $n\in\mathbb{N}$.  Write $\Theta^s$ for the set of all such probabilities.
\item [iii)] We can also show that $\Theta^s$ is big enough:
\begin {itemize}
\item [$\bullet$] $\mathbb{E}[\phi(\xi_1, \cdots, \xi_m)]=\sup\limits_{P\in\Theta^s} E_P[\phi(\xi_1, \cdots, \xi_m)]$, for any $m\in\mathbb{N}$ and $\phi\in C_{b,Lip}(\mathbb{R}^m)$;
\item [$\bullet$] $\bar{\xi}_n$ converges to a constant $m_P$ under each $P\in\Theta^s$, and $\{m_P\}_{P\in\Theta^s}$ is dense in $[\underline{\mu},\overline{\mu}]$. 
\end {itemize}
\end {itemize}
}

$\textbf{2.}$ \emph{If $\{\xi_k\}$ is $<$i.i.d under the sublinear expectation $\mathbb{E}$, i.e., $\xi_k$ is independent of $\xi_{k+1}, \cdots, \xi_{k+l}$, for any $k, l\in\mathbb{N}$, the tail $\sigma$-algebra is trivial under each $P\in\Theta^e$.
}

We refer to \cite {MM05}, \cite {Chen16} and \cite {CHW19} (and the references therein) for earlier results related to the strong law of large numbers for capacities. \cite {MM05} and \cite {Chen16} presented bounds of the cluster points similar to (\ref {sLLN-bounds-intro}) with notions of independence different from that in this paper. To define the independence between two random vectors, \cite {MM05} (resp. \cite {Chen16}) took the collection of indicator functions (resp. Borel measurable functions) as test functions. In this paper, test functions are chosen as bounded and continuous functions. \cite {CHW19} characterized the distribution of the cluster points in the interval $[\underline{\mu},\overline{\mu}]$ under the assumption that the capacities are continuous from above along Borel sets. This assumption is not needed in our proofs. The assumptions imposed in this paper are consistent with the $G$-framework introduced by Peng S. (\cite{P08b}, \cite{Peng-G}, \cite{P08a}, \cite{Peng-book}). We also refer to 
\cite {Zhang21} and \cite {GL21} for some recent results on the sLLN*.

\section {Basic notions of sublinear expectations}
Similar to linear expectations, the general framework of sublinear expectations is defined on measurable spaces.
\subsection {Sublinear expectations on measurable spaces}
Let $(\Omega, \mathcal{F})$ be a given measurable space and let $\Theta$ be a set of probabilities on $(\Omega, \mathcal {F})$.  Let $\mathcal{H}$ be a linear subspace of the space of $\mathcal {F}$-measurable functions on $\Omega$ with $\sup_{P\in\Theta} E_P[|\xi|]<\infty$.
Set $\mathbb{E}[\xi]=\sup_{P\in\Theta} E_P[\xi]$ for $\xi\in\mathcal{H}$. We call $(\Omega, \mathcal {H}, \mathbb{E})$ a sublinear expectation space.

\begin {definition} Let  $(\Omega, \mathcal{H}, \mathbb{E})$ be a sublinear expectation space. We say a random vector $\mathbf{ X}=(X_1,\cdots, X_m)\in\mathcal{H}^m$ is independent of $\mathbf{ Y}=(Y_1,\cdots, Y_n)\in\mathcal{H}^n$ if for any $\varphi\in C_{b,Lip}(\mathbb{R}^{n+m})$
\[\mathbb{E}[\varphi(\mathbf{ Y},\mathbf{ X})]=\mathbb{E}[\mathbb{E}[\varphi(\mathbf{ y}, \mathbf{ X})]|_{\mathbf{ y}=\mathbf{ Y}}].\]
\end {definition}

\begin {definition} Let $(\Omega, \mathcal{H}, \mathbb{E})$ and   $(\widetilde{\Omega}, \mathcal{\widetilde{H}}, \widetilde{\mathbb{E}})$ be two sublinear expectation spaces. A random vector $\mathbf{ X}$ in  $(\Omega, \mathcal{H}, \mathbb{E})$ is said to be identically distributed with another random vector $\mathbf{ Y}$ in $(\widetilde{\Omega}, \mathcal{\widetilde{H}}, \widetilde{\mathbf{E}})$ (write $\mathbf{ X}\overset{d}{=}\mathbf{ Y}$), if for any bounded and Lipschitz function $\varphi$, \[\mathbb{E}[\varphi(X)]=\widetilde{\mathbb{E}}[\varphi(Y)].\]

\end {definition}

In a sublinear expectation space,  the fact that $\mathbf{ X}$ is independent of $\mathbf{ Y}$ does not imply that $\mathbf{Y}$ is independent of $\mathbf{X}$.  A sequence of  random variables $(X_i)_{i\ge1}$ is called  i.i.d  if $X_{i+1}$ is independent from $(X_1,\cdots, X_i)$ and $X_{i+1}\overset{d}{=}X_i$ for each $i\in\mathbb{N}$, and called  $<$i.i.d  if $X_{i}$ is independent from $(X_{i+1}, \cdots, X_{i+j})$ and $X_{i+1}\overset{d}{=}X_i$ for each $i, j\in\mathbb{N}$. It is easy to see that these two notions do not make any difference for the weak law of large numbers. 

\subsection {Sublinear expectations on Polish spaces}

Some typical sublinear expectations, such as $G$-expectation,  are defined on complete separable metric spaces. Generally, let $\Omega$ be a Polish space. Denote by $C_b(\Omega)$ the collection of bounded and continuous functions and $\mathcal {M}_1(\Omega)$ the collection of Borel probability measures on $\Omega$.

For $\Theta\subset \mathcal {M}_1(\Omega)$, the associated sublinear expectation is defined by
\[\mathbb{E}[\xi]=\sup_{P\in\Theta}E_P[\xi], \ \xi \in C_b(\Omega).\]
We call $\Theta$ a set that represents $\mathbb{E}$.
For $p\ge1$, set $\|\xi\|_{p}:=\big(\mathbb{E}[|\xi|^p]\big)^{1/p}$, $\xi\in C_b(\Omega)$. Denote by $L^p_{\mathbb{E}}(\Omega)$ the completion of $C_b(\Omega)$ with respect to the norm $\|\cdot\|_p$. 

\begin {definition}  A sublinear expectation  $\mathbb{E}[\cdot]$ defined on a Polish space $\Omega$ is said to be regular if for each $\left\{X_{n}\right\}_{n=1}^{\infty}$ in $C_{b}(\Omega)$ such that $X_{n} \downarrow 0$ on $\Omega$,  we have $\mathbb{E}\left[X_{n}\right]\downarrow 0$.
\end {definition}

 It follows from Theorem 12 in \cite {DHP11} that $\mathbb{E}[\cdot]$ is regular if and only if $\hat{\Theta}$ is relatively compact. We also refer to \cite {DHP11} for the properties of the space $L^p_{\mathbb{E}}(\Omega)$.
 
\section {A strong law of large numbers under sublinear expectations}
 Let $\{\xi_k\}_{k\ge 1}$ be a sequence of independent and identically distributed random variables  under a sublinear expectation $\mathbb{E}=\sup_{P\in\Theta}E_P$. Set $\bar{\xi}_n=\frac{\xi_1+\cdots+\xi_n}{n}$, $\underline{\mu}=-\mathbb{E}[-\xi_1]$, $\overline{\mu}=\mathbb{E}[\xi_1]$. 
 
 \cite {Chen16} and \cite {CHW19} proved that the cluster points of $\{\bar{\xi}_n\}$ fall into $[\underline{\mu}, \overline{\mu}]$. First, we shall give a new proof to this result  based on the error estimates for wLLN* (\cite {Song19}). Then, we investigate under which conditions the collection of the cluster points of $\{\bar{\xi}_n\}$ coincides with the interval $[\underline{\mu},\overline{\mu}]$, which is not always the case shown by some counterexamples. 
 
 For a sequence of i.i.d random variables $\{\xi_k\}_{k\ge 1}$ under a sublinear expectation $\mathbb{E}=\sup_{P\in\Theta}E_P$ defined on a Polish space $\Omega$, we  show that each $\mu\in [\underline{\mu},\overline{\mu}]$ is a cluster point of the empirical averages supposing that the random variables $\{\xi_k\}$ are contained in $L^{\gamma}_\mathbb{E}(\Omega)$ for some $\gamma>1$  and $\Theta$ is weakly compact. 

 \subsection {The bounds of cluster points}
Below, we give an alternative proof to the bounds of the cluster points given in  \cite {Chen16} and \cite {CHW19}.

\begin {theorem}\label {LLN-bound} Let $\{\xi_k\}$ be a sequence of independent and identically distributed random variables under the sublinear expectation $\mathbb{E}=\sup\limits_{P\in\Theta}E_{P}$.  Suppose $\mathbb{E}[|\xi_1|^{1+\beta}]<\infty$ for some $\beta>0$. Set $\bar{\xi}_n=\frac{\xi_1+\cdots+\xi_n}{n}$. Then, for any $P\in\Theta$,  \[\underline{\mu}\le\liminf_{n\rightarrow\infty}\bar{\xi}_n\le\limsup_{n\rightarrow\infty}\bar{\xi}_n\le\overline{\mu}, \ P\emph{-a.s.,}\] where $\overline{\mu}=\mathbb{E}[\xi_1]$ and $\underline{\mu}=-\mathbb{E}[-\xi_1].$
\end {theorem}
\begin {proof}  For $m,n\in\mathbb{N}$, set $\bar{\xi}_n^{(m)}=\frac{\xi_{m+1}+\cdots+\xi_{m+n}}{n}$.  Let $e_n=[\alpha^n]$ for some $\alpha>1$ and let $\delta_n$ be a sequence of  positive integers. For $P\in \Theta$ and $\varepsilon>0$, it follows from Theorem 5.1 in \cite {Song19} that
\begin{equation}
\begin{aligned} \sum_{n=1}^{\infty} P\left\{\bar{\xi}^{(\delta_n)}_{e_n}-\overline{\mu}>\varepsilon\right\} & \le \frac{1}{\varepsilon} \sum_{n=1}^{\infty}\mathbb{E}[(\bar{\xi}^{(\delta_n)}_{e_n}-\overline{\mu})^+]\le \frac{c\mathbb{E}[|\xi_1|^{1+\beta}]}{\varepsilon} \sum_{n=1}^{\infty} \frac{1}{e_{n}^{\beta/2}} <\infty,
\end{aligned}
\end{equation}
which means that $\limsup\limits_{n\rightarrow\infty}\bar{\xi}^{(\delta_n)}_{e_n}\le\overline{\mu}, \ P\emph{-a.s.}$
Similarly, we can get $\liminf\limits_{n\rightarrow\infty}\bar{\xi}^{(\delta_n)}_{e_n}\ge\underline{\mu}, \ P\emph{-a.s.}$

For any $n\in\mathbb{N}$ large enough, there exists $s_n\in\mathbb{N}$ such that $e_{s_n}\le n< e_{s_n+1}$.

\noindent {\bf STEP 1.} The $\xi_n\ge0$ case

 Noting that $\bar{\xi}_n\le\frac{e_{s_n+1}}{e_{s_{n}}}\bar{\xi}_{e_{s_n+1}},$
we have $\limsup\limits_{n\rightarrow\infty}\bar{\xi}_{n}\le \alpha\overline{\mu}, \ P\emph{-a.s.}$
Letting $\alpha\downarrow1$
we have
\[\limsup_{n\rightarrow\infty}\bar{\xi}_{n}\le \overline{\mu}, \ P\emph{-a.s.}\] By similar arguments, we can prove that, for any $P\in\Theta$,
\[\liminf_{n\rightarrow\infty}\bar{\xi}_{n}\ge \underline{\mu}, \ P\emph{-a.s.}\]

\noindent {\bf STEP 2.}  The general case

It follows from STEP 1 that, for any $P\in\Theta$,
\[\limsup_{n\rightarrow\infty}\frac{|\xi_1|+\cdots+|\xi_n|}{n}\le \mathbb{E}[|\xi_1|], \ P\emph{-a.s.}\]
Noting that
\[ |\bar{\xi}_n|\le \frac{|\xi_1|+\cdots+|\xi_n|}{n},\]
we have
\[\limsup_{n\rightarrow\infty}|\bar{\xi}_n|\le \mathbb{E}[|\xi_1|], \ P\emph{-a.s.}\]
So $\xi^*_\infty:=\sup_{n}|\bar{\xi}_n|<\infty$, $P$-a.s. Noting that
\begin {eqnarray*}\bar{\xi}_n&=&\frac{\xi_{1}+\cdots+\xi_{n-e_{s_n}}}{n}+\frac{\xi_{n-e_{s_n}+1}+\cdots+\xi_{n}}{n}\\
&=&\frac{n-e_{s_n}}{n}\bar{\xi}_{n-e_{s_n}}+\frac{e_{s_n}}{n}(\bar{\xi}^{(n-e_{s_n})}_{e_{s_n}}-\overline{\mu})+\frac{e_{s_n}-n}{n}\overline{\mu}+\overline{\mu}\end {eqnarray*}
and $|\frac{n-e_{s_n}}{n}\bar{\xi}_{n-e_{s_n}}|\le\frac{n-e_{s_n}}{n}\xi^*_{\infty}$, $\limsup\limits_{n\rightarrow\infty}\frac{n-e_{s_n}}{n}\le\frac{\alpha-1}{\alpha}$, we have
\[\limsup_{n\rightarrow\infty}\bar{\xi}_n\le \frac{\alpha-1}{\alpha}(\xi^*_{\infty}+|\overline{\mu}|)+\overline{\mu}, \ P\emph{-a.s.}\]
Letting $\alpha\downarrow 1$, we get
\[\limsup_{n\rightarrow\infty}\bar{\xi}_n\le \overline{\mu}, \ P\emph{-a.s.}\]
On the other hand, noting that
\begin {eqnarray*}\bar{\xi}_n=\frac{n-e_{s_n}}{n}\bar{\xi}_{n-e_{s_n}}+\frac{e_{s_n}}{n}(\bar{\xi}^{(n-e_{s_n})}_{e_{s_n}}-\underline{\mu})+\frac{e_{s_n}-n}{n}\underline{\mu}+\underline{\mu}\end {eqnarray*}
 we have
\[\liminf_{n\rightarrow\infty}\bar{\xi}_n\ge -\frac{\alpha-1}{\alpha}(\xi^*_{\infty}+|\underline{\mu}|)+\underline{\mu}, \ P\emph{-a.s.}\]
and consequently, we get
\[\liminf_{n\rightarrow\infty}\bar{\xi}_n\ge \underline{\mu}, \ P\emph{-a.s.}\]
\end {proof}

\subsection {Convergence of empirical averages to cluster points}

In this subsection, we investigate under which congditions the mean interval $[\underline{\mu},\overline{\mu}]$ is filled up with cluster points of the empirical averages. First, we present several lemmas.

For $\kappa\in\{\underline{\mu}, \overline{\mu}\}^{\mathbb{N}}$, set $\mu_n=\frac{\kappa_1+\cdots+\kappa_n}{n}$, and designate $N^+_{\kappa}(n)=\{i\le n \ | \ \kappa_i=\overline{\mu}\}$,  $N^-_{\kappa}(n)=\{i\le n \ | \ \kappa_i=\underline{\mu}\}$,  and $k_n, l_n$  the numbers of elements in $N^+_{\kappa}(n)$ and $N^-_{\kappa}(n)$, respectively.

\begin {lemma} \label {lemma-LLN-esti} Let $\{\xi_k\}$ be a sequence of random variables in a sublinear expectation space $(\Omega, \mathcal{H}, \mathbb{E})$ with $\mathbb{E}=\sup_{P\in\Theta}E_P$. Set $\bar{\xi}_n=\frac{\xi_1+\cdots+\xi_n}{n}$. Then, for any $P\in\Theta$ with $E_P[\xi_n]=\kappa_n$, we have, with the convention $0/0=0$, 
\[P[|\bar{\xi}_n-\mu_n|>\varepsilon]\le \frac{4}{\varepsilon}\frac{k_n}{n}\mathbb{E}[(\frac{\Sigma_{i\in N^+_{\kappa}(n)}\xi_i}{k_n}-\overline{\mu})^+]+\frac{4}{\varepsilon}\frac{l_n}{n}\mathbb{E}[(\frac{\Sigma_{i\in N^-_{\kappa}(n)}\xi_i}{l_n}-\underline{\mu})^-].\]

\end {lemma}
\begin {proof} \begin {eqnarray*}
P[|\bar{\xi}_n-\mu_n|>\varepsilon]&=&P[|\frac{k_n}{n}(\frac{\Sigma_{i\in N^+_{\kappa}(n)}\xi_i}{k_n}-\overline{\mu})+\frac{l_n}{n}(\frac{\Sigma_{i\in N^-_{\kappa}(n)}\xi_i}{l_n}-\underline{\mu})|>\varepsilon]\\
&\le& P[\frac{k_n}{n}|\frac{\Sigma_{i\in N^+_{\kappa}(n)}\xi_i}{k_n}-\overline{\mu}|>\varepsilon/2]+P[\frac{l_n}{n}|\frac{\Sigma_{i\in N^-_{\kappa}(n)}\xi_i}{l_n}-\underline{\mu}|>\varepsilon/2]\\
&\le&\frac{4}{\varepsilon}\frac{k_n}{n}E_P[(\frac{\Sigma_{i\in N^+_{\kappa}(n)}\xi_i}{k_n}-\overline{\mu})^+]+\frac{4}{\varepsilon}\frac{l_n}{n}E_P[(\frac{\Sigma_{i\in N^-_{\kappa}(n)}\xi_i}{l_n}-\underline{\mu})^-]\\
&\le&\frac{4}{\varepsilon}\frac{k_n}{n}\mathbb{E}[(\frac{\Sigma_{i\in N^+_{\kappa}(n)}\xi_i}{k_n}-\overline{\mu})^+]+\frac{4}{\varepsilon}\frac{l_n}{n}\mathbb{E}[(\frac{\Sigma_{i\in N^-_{\kappa}(n)}\xi_i}{l_n}-\underline{\mu})^-].
\end {eqnarray*} The second inequality follows from Chebyshev's inequality as well as the fact that 
\[ E_P[\frac{\Sigma_{i\in N^+_{\kappa}(n)}\xi_i}{k_n}]=\overline{\mu} \ \textmd{and}  \ E_P[\frac{\Sigma_{i\in N^-_{\kappa}(n)}\xi_i}{k_n}]=\underline{\mu}.\]
\end {proof}

For $\mu\in[\underline{\mu}, \overline{\mu}]$, set
\begin {eqnarray}\label {kappa}
\begin {split}
\kappa_1(\mu)=\left\{\begin{array}{ll}{\overline{\mu}} & { \ \ \text { if } \mu\ge\frac{\overline{\mu}+\underline{\mu}}{2},} \\ {\underline{\mu}} & { \ \ \text { if } \mu<\frac{\overline{\mu}+\underline{\mu}}{2},}\end{array}\right. \hskip 0.3cm \mathrm{and} \hskip 0.3cm 
\kappa_{n+1}(\mu)=\left\{\begin{array}{ll}{\overline{\mu}} & { \ \ \text { if } \mu\ge\mu_n,} \\ {\underline{\mu}} & { \ \ \text { if } \mu<\mu_n,}\end{array}\right.
\end {split}
\end {eqnarray}
where $\mu_n=(\kappa_1+\cdots+\kappa_n)/n$.  

\begin {lemma}Let $N^+_{\kappa}(n)=\{i\le n \ | \ \kappa_i(\mu)=\overline{\mu}\}$,  $N^-_{\kappa}(n)=\{i\le n \ | \ \kappa_i(\mu)=\underline{\mu}\}$,  and $k_n, l_n$  be the numbers of elements in $N^+_{\kappa}(n)$ and $N^-_{\kappa}(n)$, respectively. Then, we have
\begin {eqnarray}\label {kappa-esti}
|\mu_n-\mu|\le\frac{\overline{\mu}-\underline{\mu}}{n}.
\end {eqnarray} Equivalently, \[|\frac{k_n}{n}-\frac{\mu-\underline{\mu}}{\overline{\mu}-\underline{\mu}}|\le \frac{1}{n}, \ |\frac{l_n}{n}-\frac{\overline{\mu}-\mu}{\overline{\mu}-\underline{\mu}}|\le \frac{1}{n}.\]
\end {lemma}
\begin {proof}
Clearly, (\ref {kappa-esti}) holds for $n=1$. Assume that (\ref{kappa-esti}) holds for $i\le n$.
For the $\mu\ge\mu_n$ case, \[\mu_{n+1}-\mu=\frac{n}{n+1}(\mu_n-\mu)+\frac{1}{n+1}(\overline{\mu}-\mu).\] Noting that $|\frac{n}{n+1}(\mu_n-\mu)|\le\frac{1}{n+1}(\overline{\mu}-\underline{\mu})$, $|\frac{1}{n+1}(\overline{\mu}-\mu)|\le\frac{1}{n+1}(\overline{\mu}-\underline{\mu})$ and $(\mu_n-\mu)(\overline{\mu}-\mu)\le0$, we have
\[|\mu_{n+1}-\mu|\le\frac{1}{n+1}(\overline{\mu}-\underline{\mu}).\] The $\mu<\mu_n$ case can be proved similarly.  \end {proof}

\begin {lemma} \label {lemma-LLN-inP} Let $\{\xi_k\}$ be a sequence of i.i.d random variables in a sublinear expectation space $(\Omega, \mathcal{H}, \mathbb{E})$ with $\mathbb{E}=\sup_{P\in\Theta}E_P$. Set $\bar{\xi}_n=\frac{\xi_1+\cdots+\xi_n}{n}$. Assume $\mathbb{E}[|\xi_1|^{1+\beta}]<\infty$ for some $\beta>0$. For $\mu\in[\underline{\mu}, \overline{\mu}]$ and $P\in\Theta$ with $E_P[\xi_n]=\kappa_n(\mu)$, $n\in\mathbb{N}$, we have, for any $\varepsilon>0$,
\[P[|\bar{\xi}_n-\mu|>\varepsilon]\rightarrow0\]
as $n$ goes to $+\infty$.

\end {lemma}

\begin {proof} 
It follows from Lemma \ref {lemma-LLN-esti} that
\begin {eqnarray}\label {Esti-m}
P[|\bar{\xi}_n-\mu_n|>\varepsilon]\le \frac{4}{\varepsilon}\frac{k_n}{n}\mathbb{E}[(\frac{\Sigma_{i\in N^+_{\kappa}(n)}\xi_i}{k_n}-\overline{\mu})^+]+\frac{4}{\varepsilon}\frac{l_n}{n}\mathbb{E}[(\frac{\Sigma_{i\in N^-_{\kappa}(n)}\xi_i}{l_n}-\underline{\mu})^-],
\end {eqnarray}
which converges to 0 as $n$ goes to infinity by the wLLN* in \cite {P08b}.  Note that
\[P[|\bar{\xi}_n-\mu|>\varepsilon]\le P[|\bar{\xi}_n-\mu_n|>\varepsilon-|\mu_n-\mu|].\] Since $|\mu_n-\mu|\le\frac{\overline{\mu}-\underline{\mu}}{n}$, we have,  for $n$ large enough, $|\mu_n-\mu|<\frac{\varepsilon}{2}$, and consequently,
\[P[|\bar{\xi}_n-\mu|>\varepsilon]\le P[|\bar{\xi}_n-\mu_n|>\varepsilon/2].\]
The desired result follows from (\ref {Esti-m}).
\end {proof}

Lemma \ref {lemma-LLN-inP} proves the convergence in probability based on the weak law of large numbers under sublinear expectations (wLLN*). By the error estimates of wLLN* given by \cite{Song19}, we can prove that $\bar{\xi}_n$ converges to $\mu$ almost surely under $P\in\Theta$ satisfying the property in Lemma \ref {lemma-LLN-inP}.

\begin {lemma} \label {lemma-LLN-as} Assume the conditions in Lemma \ref {lemma-LLN-inP} hold. For $\mu\in[\underline{\mu}, \overline{\mu}]$ and $P\in\Theta$ with $E_P[\xi_n]=\kappa_n(\mu)$, $n\in\mathbb{N}$, we have \[\bar{\xi}_n\rightarrow \mu, \ P\emph{-a.s.}\]
\end {lemma}
\begin {proof} 
For $\mu\in[\underline{\mu}, \overline{\mu}]$, note that
\[\bar{\xi}_n=\frac{k_n}{n}(\frac{1}{k_n}\Sigma_{i\in N^+_{\kappa}(n)}\xi_i)+\frac{l_n}{n}(\frac{1}{l_n}\Sigma_{i\in N^-_{\kappa}(n)}\xi_i).\]
Set $N^+_{\kappa}=\bigcup_nN^+_{\kappa}(n)$ and $N^-_{\kappa}=\bigcup_nN^-_{\kappa}(n)$. 

Let $e_n=[\alpha^n]$ for some $\alpha>1$. By Lemma \ref {lemma-LLN-esti}, we have
\begin{eqnarray*}
& & \sum_{n=1}^{\infty} P\left\{\left|\bar{\xi}_{e_n}-\mu_n\right|>\varepsilon\right\} \\
&\le& \frac{4}{\varepsilon} \sum_{n=1}^{\infty} \frac{k_{e_n}}{e_n}\mathbb{E}[(\frac{\Sigma_{i\in N^+_{\kappa}(e_n)}\xi_i}{k_{e_n}}-\overline{\mu})^+]+\frac{4}{\varepsilon} \sum_{n=1}^{\infty} \frac{l_{e_n}}{e_n}\mathbb{E}[(\frac{\Sigma_{i\in N^-_{\kappa}(e_n)}\xi_i}{l_{e_n}}-\underline{\mu})^-].
\end{eqnarray*} 
It follows from Theorem 5.1 in \cite {Song19} that
\begin{eqnarray*}
& &\frac{4}{\varepsilon} \sum_{n=1}^{\infty} \frac{k_{e_n}}{e_n}\mathbb{E}[(\frac{\Sigma_{i\in N^+_{\kappa}(e_n)}\xi_i}{k_{e_n}}-\overline{\mu})^+]\\
&\le& \frac{c\mathbb{E}[|\xi_1|^{1+\beta}]}{\varepsilon}\sum_{n=1}^{\infty}\frac{k_{e_n}}{e_n}{k_{e_n}}^{-\beta/(1+\beta)}\le\frac{c\mathbb{E}[|\xi_1|^{1+\beta}]}{\varepsilon}\sum_{n=1}^{\infty}{e_n}^{-\beta/(1+\beta)}<\infty.
\end{eqnarray*}
Similarly, we get 
\[\frac{4}{\varepsilon} \sum_{n=1}^{\infty} \frac{l_{e_n}}{e_n}\mathbb{E}[(\frac{\Sigma_{i\in N^-_{\kappa}(e_n)}\xi_i}{l_{e_n}}-\underline{\mu})^-]<\infty,\] and consequently, 
\[\sum_{n=1}^{\infty} P\left\{\left|\bar{\xi}_{e_n}-\mu_n\right|>\varepsilon\right\}<\infty,\]
which implies that  \[\lim_{n\rightarrow\infty}\bar{\xi}_{e_n}=\mu, \ P\emph{-a.s.}\]
For any $n\in\mathbb{N}$ large enough, there exists $s_n\in\mathbb{N}$ such that $e_{s_n}\le n< e_{s_n+1}$, and
\begin {eqnarray*}
\bar{\xi}_n=\frac{n-e_{s_n}}{n}\bar{\xi}_{n-e_{s_n}}+\frac{e_{s_n}}{n}(\bar{\xi}^{(n-e_{s_n})}_{e_{s_n}}-\mu)+\frac{e_{s_n}-n}{n}\mu+\mu.
\end {eqnarray*}
By Theorem \ref {LLN-bound}, we have $\xi^*_{\infty}=\sup_{n}|\bar{\xi}_n|<\infty$, $P$-a.s.
Since $\lim\limits_{n\rightarrow\infty}\frac{e_{s_n}}{n}(\bar{\xi}^{(n-e_{s_n})}_{e_{s_n}}-\mu)=0$, $P$-a.s. and  $\limsup\limits_{n\rightarrow\infty}\big(|\frac{n-e_{s_n}}{n}\bar{\xi}_{n-e_{s_n}}|+|\frac{e_{s_n}-n}{n}\mu|\big)\le \frac{\alpha-1}{\alpha}(\xi^*_{\infty}+|\mu|)$,  we have \[\limsup_{n\rightarrow\infty}|\bar{\xi}_n-\mu|\le \frac{\alpha-1}{\alpha}(\xi^*_{\infty}+|\mu|), \ P\emph{-a.s.}\]
Letting $\alpha\downarrow 1$, we get
\[\lim_{n\rightarrow\infty}\bar{\xi}_n= \mu, \ P\emph{-a.s.}\]
\end {proof}

In the sequel of this subsection, we consider sublinear expectations on a Polish space $\Omega$, and show that each $\mu\in [\underline{\mu},\overline{\mu}]$ is a cluster point of the empirical averages supposing that the random variables $\{\xi_n\}$ are quasi-continuous and $\Theta$ is weakly compact. 

Let $\Omega$ be a Polish space. Denote by $C_b(\Omega)$ the collection of bounded and continuous functions and $\mathcal {M}_1(\Omega)$ the collection of Borel probability measures on $\Omega$. For $\Theta\subset \mathcal {M}_1(\Omega)$ which is weakly compact, the associated sublinear expectation is defined by
\[\mathbb{E}[\xi]=\sup_{P\in\Theta}E_P[\xi], \ \xi \in C_b(\Omega).\]

\begin {theorem} \label {LLN-P-kappa} Let $\{\xi_k\}\subset L^{1+\beta}_{\mathbb{E}}(\Omega)$, $\beta>0$, be a sequence of independent and identically distributed random variables under $\mathbb{E}$. Set $\underline {\mu}=-\mathbb{E}[-\xi_1]$, $\overline {\mu}=\mathbb{E}[\xi_1]$ and $\bar{\xi}_n=\frac{\xi_1+\cdots+\xi_n}{n}$. Then, for any $\mu\in[\underline {\mu}, \overline {\mu}]$, there exists $P_\mu\in\Theta$ such that \[E_{P_{\mu}}[\xi_n]=\kappa_n(\mu).\] Particularly, 
\[\lim_n\bar{\xi}_n=\mu, \ P_{\mu}\textmd{-a.s.}\]
\end {theorem}

\begin {proof} 
Set
\begin {eqnarray}\label {sigma(mu)}
\begin {split}
\sigma_{\mu}(n)=\left\{\begin{array}{ll}{1} & { \ \ \text { if }  \kappa_n(\mu)=\overline{\mu},} \\ {-1} & { \ \ \text { if } \kappa_n(\mu)=\underline{\mu}.}\end{array}\right. \end {split}
\end {eqnarray}
For $n\in\mathbb{N}$, choose $P_n\in\Theta$ such that
\[E_{P_n}[\Sigma_{k=1}^n\sigma(k)\xi_k]=\mathbb{E}[\Sigma_{k=1}^n\sigma(k)\xi_k].\]
Then, for $k\le n$,  we have
\[E_{P_n}\left[\xi_{k}\right]=\left\{\begin{array}{ll}{\overline{\mu}} & { \ \ \text { if } \sigma(k)=1,} \\ {\underline{\mu}} & { \ \ \text { if } \sigma(k)=-1.}\end{array}\right.\]
Since $\Theta$ is weakly compact, we can find a subsequent $\{P_{n_k}\}$ of $\{P_{n}\}$ and $P_{\mu}\in\Theta$ such that $P_{n_k}$ converges to $P_{\mu}$ weakly. Then, for any $n\in \mathbb{N}$,
\[E_{P_\mu}\left[\xi_{n}\right]=\lim_{k\rightarrow\infty}E_{P_{n_k}}\left[\xi_{n}\right]=\left\{\begin{array}{ll}{\overline{\mu}} & { \ \ \text { if } \sigma(n)=1,} \\ {\underline{\mu}} & { \ \ \text { if } \sigma(n)=-1.}\end{array}\right.\]
In other words, $E_{P_{\mu}}[\xi_n]=\kappa_n(\mu)$. The desired result follows from Lemma \ref {lemma-LLN-as}. 
\end {proof}

We have the following straightforward corollary.
\begin {corollary}Let $\{\xi_k\}\subset L^{1+\beta}_{\mathbb{E}}(\Omega)$ for some $\beta>0$ be a sequence of independent and identically distributed random variables under $\mathbb{E}$. Set $\underline {\mu}=-\mathbb{E}[-\xi_1]$, $\overline {\mu}=\mathbb{E}[\xi_1]$. Then, both $\xi_{\infty}^{\triangledown}:=\limsup\limits_{n\rightarrow\infty}\frac{\xi_1+\cdots+\xi_n}{n}$ and $\xi_{\infty}^{\vartriangle}:=\liminf\limits_{n\rightarrow\infty}\frac{\xi_1+\cdots+\xi_n}{n}$ are maximally distributed with means $\underline {\mu}$ and $\overline {\mu}$. Precisely, for $\phi\in C_{b,Lip}(\mathbb{R})$,
\[\mathbb{E}[\phi(\xi_{\infty}^{\triangledown})]=\mathbb{E}[\phi(\xi_{\infty}^{\vartriangle})]=\max_{\mu\in[\underline {\mu}, \overline {\mu}]}\phi(\mu).\]

\end {corollary}

\begin {lemma} \label {BorelMap} Under the same conditions as those in Theorem \ref {LLN-P-kappa}, there exists a Borel mapping
$P_{\mu}:\ [\underline{\mu},\overline{\mu}]\rightarrow\Theta$ such that $E_{P_{\mu}}[\xi_n]=\kappa_n(\mu).$
\end {lemma}

\begin {proof}  Firstly, it is easy to see that the mapping $\kappa(\mu): [\underline{\mu},\overline{\mu}]\rightarrow \{\underline{\mu}, \overline{\mu}\}^{\mathbb{N}}$ defined in (\ref {kappa}) is Borel measurable. Here, the distance on $\{\underline{\mu}, \overline{\mu}\}^{\mathbb{N}}$ is given by $d(\kappa, \kappa')=\Sigma_{n=0}^\infty \frac{1}{2^n}|\kappa(n)-\kappa'(n)|$. For $\kappa\in \{\underline{\mu}, \overline{\mu}\}^{\mathbb{N}}$, set 
\begin {eqnarray}\label {K(sigma)}
K_{\kappa}=\{P\in\Theta \ | \ E_{P}[\xi_n]=\kappa_n, n\in\mathbb{N}.\}
\end {eqnarray}  So it suffices to find a Borel mapping $P_{\kappa}: \{-1,1\}^{\mathbb{N}}\rightarrow\Theta$ such that $P_{\kappa}\in K_{\kappa}$. 

 Assume that $d(\kappa^m, \kappa)\rightarrow0$, $P_{\kappa^m}\in K_{\kappa^m}$, and that $P_{\kappa^{m_k}}$ converges weakly to $P$. Then
\[E_P[\xi_n]=\lim_{k}E_{P_{\kappa^{m_k}}}[\xi_n]=\lim_{k}\kappa^{m_k}_n=\kappa_n,\] which says that $P\in K_{\kappa}$. Then, by Lemma 12.1.8 and Theorem 12.1.10 in Stroock and Varadhan (1997) (\cite {SV97}), we can find a Borel measurable mapping $P_{\kappa}: \{\underline{\mu}, \overline{\mu}\}^{\mathbb{N}}\rightarrow\Theta$ such that $P_{\kappa}\in K_{\kappa}$. 
\end {proof}

\vskip 0.1 cm

A key step in the proof to Lemma \ref{lemma-LLN-as} is to show that, for any $\varepsilon>0$ and $\gamma>1$,
\begin {eqnarray}\sum_{n=1}^{\infty} \label {Borel-Cantelli} P_{\overline{\mu}}\left\{\left|\bar{\xi}_{e_n}-\overline{\mu}\right|>\varepsilon\right\}\le \frac{c\mathbb{E}[|\xi_1|^{\gamma}]}{\varepsilon} \sum_{n=1}^{\infty} \frac{1}{e_{n}^{\frac{\gamma-1}{\gamma}}} ,
\end {eqnarray} for $P_{\overline{\mu}}\in\Theta$ such that $E_{P_{\overline{\mu}}}[\xi_n]=\overline{\mu}$, which is based on the error estimates for the wLLN* given in \cite {Song19}. For $\alpha=2$, (\ref {Borel-Cantelli}) can also be proved by Lemma \ref {irrelevant} below.
In fact, Chebyshev's inequality gives
\begin{equation*}
\begin{aligned} \sum_{n=1}^{\infty} P_{\overline{\mu}}\left\{\left|\bar{\xi}_{e_n}-\overline{\mu}\right|>\varepsilon\right\} & \le \frac{1}{\varepsilon^2} \sum_{n=1}^{\infty} E_{P_{\overline{\mu}}}[(\bar{\xi}_{e_n}-\overline{\mu})^2].
\end{aligned}
\end {equation*}
It follows from Lemma \ref {irrelevant} that
\begin {eqnarray*}
E_{P_{\overline{\mu}}}[(\bar{\xi}_{e_n}-\overline{\mu})^2]&=&\Sigma_{k=1}^{e_n}\frac{1}{e_n^2}E_{P_{\overline{\mu}}}[(\xi_k-\overline{\mu})^2]+\Sigma_{i<j}\frac{2}{e_n^2}E_{P_{\overline{\mu}}}[(\xi_i-\overline{\mu})(\xi_j-\overline{\mu})]\\
&\le&\Sigma_{k=1}^{e_n}\frac{1}{e_n^2}E_{P_{\overline{\mu}}}[\xi_k^2]\le \frac{\mathbb{E}[|\xi_1|^{2}]}{e_n}.
\end {eqnarray*}
Therefore, we have
\begin{equation*}
\begin{aligned} \sum_{n=1}^{\infty} P_{\overline{\mu}}\left\{\left|\bar{\xi}_{e_n}-\overline{\mu}\right|>\varepsilon\right\} \le \frac{\mathbb{E}[|\xi_1|^{2}]}{\varepsilon^2} \sum_{n=1}^{\infty} \frac{1}{e_{n}} <\infty.
\end{aligned}
\end {equation*}
\begin {lemma} \label {irrelevant}Suppose that $\eta$ is independent to $\xi$ under $\mathbb{E}$ for $\xi, \eta\in L^2_{\mathbb{E}}(\Omega)$. Then, for $P\in\Theta$ with $E_P[\xi+\eta]=\mathbb{E}[\xi+\eta]$, we have
\[E_P[\xi\eta]\le E_P[\xi]E_P[\eta].\]
\end {lemma}
\begin {proof} If $\xi\ge0$, we have, for any $Q\in\Theta$,
\begin {eqnarray}\label {indep-p}
E_Q[\xi\eta]\le\mathbb{E}[\xi\eta]=\mathbb{E}[\xi]\mathbb{E}[\eta]=E_P[\xi]E_P[\eta].
\end {eqnarray}
If $\xi\ge-c$ for some $c\ge0$, it follows from (\ref {indep-p}) that, for any $Q\in\Theta$,
\[E_Q[(\xi+c)\eta]\le E_P[\xi+c]E_P[\eta],\]
or, equivalently,
\begin {eqnarray}\label {indep-bb}
E_Q[\xi\eta]\le E_P[\xi]E_P[\eta]+c(E_P[\eta]-E_Q[\eta]).
\end {eqnarray}
For the general case, set $\xi_n=\max\{\xi,-n\}$. Choose $P_n\in \Theta$ such that $E_{P_n}[\xi_n+\eta]=\mathbb{E}[\xi_n+\eta]$. Note that $E_{P_n}[\eta]=E_P[\eta]=\mathbb{E}[\eta]$. It follows from (\ref {indep-bb}) that
\[E_P[\xi_n\eta]\le E_{P_n}[\xi_n]E_{P_n}[\eta]+n(E_{P_n}[\eta]-E_P[\eta]),\]
and consequently, we have $E_P[\xi_n\eta]\le \mathbb{E}[\xi_n]\mathbb{E}[\eta].$
Letting $n$ go to infinity, we have \[E_P[\xi\eta]\le \mathbb{E}[\xi]\mathbb{E}[\eta]=E_P[\xi]E_P[\eta].\]
\end {proof}
\begin {remark} In the above lemma, if $E_P[\xi+\eta]=-\mathbb{E}[-(\xi+\eta)]$, we also have
\[E_P[\xi\eta]\le E_P[\xi]E_P[\eta],\]
and, if $E_P[\xi-\eta]=\mathbb{E}[\xi-\eta]$ or $E_P[\eta-\xi]=\mathbb{E}[\eta-\xi]$
we also have
\[E_P[\xi\eta]\ge E_P[\xi]E_P[\eta].\]
\end {remark}

\section {On the tail $\sigma$-algebra}
In a probability space, it is well-known by the Kolmogorov 0-1 law that the tail $\sigma$-algebra of a sequence of independent and identically distributed random variables is trivial. In this section, we study the triviality of the tail $\sigma$-algebra of a sequence of i.i.d  random variables under a sublinear expectation.

Let $\Omega=\mathbb{R}^{\mathbb{N}}$ endowed with the metric $d(x,y):=\Sigma_{k=1}^{\infty}\frac{1}{2^n}(|x(k)-y(k)|\wedge 1)$, for $x, y \in \Omega$. This is a complete separable metric space. 

Notations:
\begin{itemize}
  \item For $k\in\mathbb{N}$, set $\xi_k(\omega)=\omega(k)$, $\omega\in\Omega$;
  \item For $n\in\mathbb{N}$, let $\theta_n\{\xi_1, \xi_2, \cdots\}=\{\xi_{n+1}, \xi_{n+2}, \cdots\}$ be the $n$-shift operator;
  \item $\mathcal{G}_n:=\sigma\{\xi_n, \xi_{n+1}, \cdots \}$, $\mathcal{F}_n:=\sigma\{\xi_1,\cdots, \xi_{n}\}$ and $\mathcal {T}:=\cap_n\mathcal{G}_n$;
  \item $\bar{\xi}_n:=\frac{\xi_1+\cdots+\xi_n}{n}$, $\xi_{\infty}^{\triangledown}:=\limsup\limits_{n\rightarrow\infty}\bar{\xi}_n$, $\xi_{\infty}^{\vartriangle}:=\liminf\limits_{n\rightarrow\infty}\bar{\xi}_n$.
  \end{itemize}
 
 For a regular sublinear expectation $\mathbb{E}$ on $(\Omega, C_b(\Omega))$, set
\begin {eqnarray*}
\Theta:=\{P\in \mathcal{M}_1(\Omega)|& & E_{P}[X]\le \hat{\mathbb{E}}[X], \ \textmd{for any} \ X\in C_b(\Omega)\},
\end {eqnarray*} and let $\Theta^e$ be the set of the extreme points of $\Theta$.  By the representation theorem of sublinear expectations, we have,  for $X\in C_b(\Omega)$, $\mathbb{E}[X]=\sup_{P\in\Theta}E_P[X].$  This representation still holds for $X\in L^p_{\mathbb{E}}(\Omega)$,  $p\ge1$, the completion of $C_b(\Omega)$ under the $L^p$-norm $\|\cdot\|_p^p=\mathbb{E}[|\cdot|^p]$. 

\vskip 0.1 cm

\noindent {\textbf{Assumption}} Throughout this section, we assume that $\{\xi_k\}\subset L^{\gamma}_{\mathbb{E}}(\Omega)$, for some $\gamma>1$.

For $d\in\mathbb{N}$ and $P\in\Theta$, set $P_{1,d}=P\circ(\xi_1,\cdots, \xi_d)^{-1}$, and $\Xi_d=\{P_{1,d} \ | \ P\in\Theta\}$, which is convex and closed. Designate $\Xi_d^e$ the set of extreme points of $\Xi_d$.

Supposing $\xi_1, \xi_2, \cdots$ is a sequence of i.i.d random variables under $\mathbb{E}$, we first give a generalization of Theorem \ref {LLN-P-kappa} in the space $(\Omega, L^1_{\mathbb{E}}(\Omega), \mathbb{E})$. Write $\overline {\mu}=\mathbb{E}[\xi_1]$, $\underline {\mu}=-\mathbb{E}[-\xi_1]$.

\begin {theorem} \label {Thm-LLN-Fd} For any $\mathcal{F}_d$-measurable random variable $\Pi=\pi(\xi_1,\cdots,\xi_d)$ with values in $[\underline{\mu}, \overline{\mu}]$ and $d\in\mathbb{N}$, and any $P\in\Theta$, there exists a probability $P^{\Pi}\in\Theta$ such that $P^{\Pi}=P$ on $\mathcal{F}_d$,  and
\begin {eqnarray}\label {LLN-Fd-as}
\lim_{n\rightarrow \infty}\bar{\xi}_n =\Pi, \ P^{\Pi}\emph{-a.s.}
\end {eqnarray}
Furthermore, if $P_{1,d}\in\Xi_d^e$, $P^{\Pi}$ can also be chosen from $\Theta^e$.
\end {theorem}
 \begin {proof} It follows from Lemma \ref {BorelMap} that there exists a Borel mapping $P_{\mu}: [\underline{\mu},\overline{\mu}]\rightarrow\Theta$ such that $E_{P_{\mu}}[\xi_n]=\kappa_n(\mu)$. Particularly, we have
 
 \[\lim_{n\rightarrow \infty}\bar{\xi}_n =\mu, \ P_{\mu}\emph{-a.s.} \]
 For any $\phi\in C_{b, Lip}(\mathbb{R}^d),  \ n\in\mathbb{N}$, set 
\begin{eqnarray} \label{P-Pi}
E_{P^{\Pi}}[\phi(\xi_1, \cdots, \xi_n)]=E_P\big[E_{P_{\pi(x_1, \cdots, x_d)}}[\phi(x_1,\cdots,x_d, \xi_{1}, \cdots, \xi_{n-d})]|^{x_i=\xi_i}_{i=1,\cdots,d}\big].
\end {eqnarray}
 Clearly, $P^{\Pi}\in\Theta$, $P^{\Pi}=P$ on $\mathcal{F}_d$ , and $P^{\Pi}(\xi_{n+d} | \mathcal{F}_d)(\omega)=P_{\Pi(\omega)}(\xi_n)=\kappa_n(\Pi(\omega))$, $P$-a.s. 
 
 We denote by $\Theta_P^{\Pi}$ the set of probabilities satisfying \[\tilde {P}^{\Pi}\in\Theta, \ \tilde{P}^{\Pi}=P\  \textmd{on} \ \mathcal{F}_d, \ \textmd{and} \ \tilde{P}^{\Pi}(\xi_{n+d} | \mathcal{F}_d)(\omega)=\kappa_n(\Pi(\omega)), \  P\textmd{-a.s.} \]  It is easily seen that $\Theta^{\Pi}_{P}$ is a nonempty convex subset of $\Theta$.  
  
We first prove (\ref {LLN-Fd-as}) holds for any $\tilde{P}^{\Pi}\in\Theta^{\Pi}_{P}$. Actually,
\begin {eqnarray*}
& &\tilde{P}^{\Pi}( \lim_{n\rightarrow \infty}\bar{\xi}_n =\Pi)\\
&=& E_{P}[\tilde{P}^{\Pi}(\lim_{n\rightarrow \infty}\bar{\xi}_n =\Pi|\mathcal{F}_d)]\\
&=& \int_{\Omega}\tilde{P}^{\Pi}(\lim_{n\rightarrow \infty}\bar{\xi}_n =\Pi(\omega)|\mathcal{F}_d)(\omega)P(d\omega).
\end {eqnarray*}

By Lemma \ref {Char-P-right} below,  we have $\tilde{P}^{\Pi}(\cdot \ | \ \mathcal{F}_d)(\omega)\circ\theta_d^{-1}\in\Theta$, which combining the fact $\tilde{P}^{\Pi}(\xi_{n+d} | \mathcal{F}_d)(\omega)=\kappa_n(\Pi(\omega)$ gives that $\tilde{P}^{\Pi}(\lim\limits_{n\rightarrow \infty}\bar{\xi}_n =\Pi(\omega)|\mathcal{F}_d)(\omega)=1$ by Lemma \ref{lemma-LLN-as}. So, $\tilde{P}^{\Pi}( \lim\limits_{n\rightarrow \infty}\bar{\xi}_n =\Pi)=1.$

Now we prove the last assertion of the theorem. We assert that $\Theta^{\Pi}_{P}$ is also closed. Actually, let $\{\tilde {P}^{\Pi, m}\}\subset \Theta^{\Pi}_{P}$ converge weakly to $\tilde{P}$. Then, for any $\phi\in C_{b, Lip}(\mathbb{R}^d)$ and $Y\in L^{\gamma}_{\mathbb{E}}(\Omega)$,  we have 
\begin {eqnarray*}\lim_mE_P\big[\phi(\xi_1,\cdots,\xi_d)E_{\tilde{P}^{\Pi, m}}[Y\circ\theta_d|\mathcal{F}_d]\big]&=&\lim_mE_{\tilde{P}^{\Pi, m}}\big[\phi(\xi_1,\cdots,\xi_d)Y\circ\theta_d\big]\\
&=&E_{\tilde{P}}\big[\phi(\xi_1,\cdots,\xi_d)Y\circ\theta_d\big]\\
&=&E_{\tilde{P}}\big[\phi(\xi_1,\cdots,\xi_d)E_{\tilde{P}}[Y\circ\theta_d|\mathcal{F}_d]\big].
\end {eqnarray*} which implies that $\tilde{P}=P$ on $\mathcal{F}_d$, and $E_{\tilde{P}^{\Pi, m}}[Y\circ\theta_d|\mathcal{F}_d]$ converges weakly to $E_{\tilde{P}}[Y\circ\theta_d|\mathcal{F}_d]$ in $L^{\gamma}(\Omega, \mathcal{F}_d, P)$ for any $Y\in L^{\gamma}_E(\Omega)$. Particularly, we have
\begin {eqnarray*}E_{P}\big[\phi(\xi_1,\cdots,\xi_d)E_{\tilde{P}}[\xi_{n+d}|\mathcal{F}_d]\big]&=&\lim_mE_{P}\big[\phi(\xi_1,\cdots,\xi_d)E_{\tilde{P}{\Pi, m}}[\xi_{n+d}|\mathcal{F}_d]\big]\\
&=&E_{P}\big[\phi(\xi_1,\cdots,\xi_d)\kappa_n(\Pi)\big].
\end {eqnarray*}
It follows that $E_{\tilde{P}}[\xi_{n+d}|\mathcal{F}_d]=\kappa_n(\Pi)$, $P$-a.s., by which we get $\tilde{P}\in\Theta^{\Pi}_{P}$.

By Krein-Milman Theorem, $\Theta^{\Pi, e}_P$ is nonempty. For $P^{\Pi}\in\Theta^{\Pi, e}_P$,  we prove that $P^{\Pi}\in\Theta^{e}$ if $P_{1,d}\in\Xi_d^e$. Assume there are $P^1, P^2\in\Theta$ and $\alpha\in(0,1)$ such that $P^{\Pi}=\alpha P^1+(1-\alpha)P^2$. Since $P_{1,d}\in\Xi_d^e$, we get $P^1=P^2=P^{\Pi}=P$ on $\mathcal{F}_d$. Then \[E_{P^{\Pi}}[\cdot|\mathcal{F}_d]=\alpha E_{P^{1}}[\cdot|\mathcal{F}_d]+(1-\alpha)E_{P^{2}}[\cdot|\mathcal{F}_d], \  P\textmd{-a.s.}\] Particularly, 
\[\kappa_n(\Pi)=E_{P^{\Pi}}[\xi_{n+d}|\mathcal{F}_d]=\alpha E_{P^{1}}[\xi_{n+d}|\mathcal{F}_d]+(1-\alpha)E_{P^{2}}[\xi_{n+d}|\mathcal{F}_d].\] 
 Since $\underline{\mu}\le E_{P^{1}}[\xi_{n+d}|\mathcal{F}_d], E_{P^{2}}[\xi_{n+d}|\mathcal{F}_d]\le\overline{\mu}$
and $\kappa_n(\Pi)=\underline{\mu}$ or $\overline{\mu}$, we get \[E_{P^{1}}[\xi_{n+d}|\mathcal{F}_d]=E_{P^{2}}[\xi_{n+d}|\mathcal{F}_d]=\kappa_n(\Pi),\] which means $P^1, P^2\in\Theta^{\Pi}_P$. So we get $P^{\Pi}=P^1=P^2$ since $P^{\Pi}\in\Theta^{\Pi, e}_P$.
 \end {proof}

\begin {corollary} \label {cor-LLN-Fd} Let $\Theta_0$ be a weakly compact subset of $\mathcal{M}_1(\Omega)$ such that $\mathbb{E}[X]=\sup_{P\in\Theta_0}E_P[X]$ for $X\in C_b(\Omega)$. Then, for  any $\mathcal{F}_d$-measurable random variable $\Pi=\pi(\xi_1,\cdots,\xi_d)$ with values in $[\underline{\mu}, \overline{\mu}]$ and $d\in\mathbb{N}$, there exists a probability $P^{\Pi}\in\Theta_0$ such that
\begin {eqnarray*}\label {LLN-Fd}
\lim_{n\rightarrow \infty}\bar{\xi}_n =\Pi, \ P^{\Pi}\textmd{-a.s.}
\end {eqnarray*}
\end {corollary}
\begin {proof} Since $\Theta_0$ represents $\mathbb{E}$, it follows from Hahn-Banach theorem that $\overline{\textmd{co}}(\Theta_0)=\Theta$. By Krein-Milman Theorem, we have $ \Theta^e\subset\Theta_0.$  For any $\mu\in\Xi_d^e$, it follows from Theorem \ref {Thm-LLN-Fd} that there exists $P^{\Pi}\in\Theta^e\subset\Theta_0$ such that $P^{\Pi}_{1,d}=\mu$ and
\begin {eqnarray*}\label {LLN-Fd}
\lim_{n\rightarrow \infty}\bar{\xi}_n =\Pi, \ P^{\Pi}\textmd{-a.s.}
\end {eqnarray*}
\end {proof}
  
  We give an example to show that for general random variable $\Pi$ Theorem \ref {Thm-LLN-Fd} and Corollary \ref {cor-LLN-Fd} may not hold. Define $\phi(x)=\underline{\mu}$ for $x\ge\frac{\underline{\mu}+\overline{\mu}}{2}$,  and $\phi(x)=\overline{\mu}$ for $x<\frac{\underline{\mu}+\overline{\mu}}{2}$. Then, it is clear that $\{\omega\in\Omega\ | \ \phi(\xi_{\infty}^{\triangledown})=\xi_{\infty}^{\triangledown}\}$ is empty.  If Theorem \ref {Thm-LLN-Fd} or Corollary \ref {cor-LLN-Fd}   holds for $\Pi=\phi(\xi_{\infty}^{\triangledown})$, then there is $P^{\Pi}\in\Theta$ such that
  \[\phi(\xi_{\infty}^{\triangledown})=\phi(\lim_n\bar{\xi}_n)=\lim_n\bar{\xi}_n=\xi_{\infty}^{\triangledown}, \ P^{\Pi}\textmd{-a.s.,}\] which is a contradiction.

\subsection {Tail $\sigma$-algebra of $<$i.i.d random variables} 

Let $\overleftarrow{\mathbb{E}}$ be a regular sublinear expectation on $(\Omega, C_b(\Omega))$ such that $\xi:=(\xi_1, \xi_2, \cdots)$ is a sequence of $<$i.i.d random variables.   

Set
\begin {eqnarray*}
\overleftarrow{\Theta}:=\{P\in \mathcal{M}_1(\Omega)|& & E_{P}[X]\le \overleftarrow{\mathbb{E}}[X], \ \textmd{for any} \ X\in C_b(\Omega)\}.
\end {eqnarray*}
It is easy to see that $\overleftarrow{\Theta}$ is weakly compact and convex. Let $\overleftarrow{\Theta}^e$ be the set of the extreme points of $\overleftarrow{\Theta}$. By Krein-Milman Theorem, one has $\overleftarrow{\Theta}=\overline{\textmd{co}}(\overleftarrow{\Theta}^e)$, which implies that \[\sup_{P\in \overleftarrow{\Theta}^e}E_P[X]=\sup_{P\in \overleftarrow{\Theta}}E_P[X]=\overleftarrow{\mathbb{E}}[X], \ \textit{for} \ X\in C_b(\Omega).\]

In this section, we show that the tail $\sigma$-algebra of $<$i.i.d random variables is $\emph{trivial}$ under each $P\in \overleftarrow{\Theta}^e$ (Theorem \ref {Triviality-left} below). To this end, we give a characterization of the set $\overleftarrow{\Theta}$. 

Set
\begin {eqnarray*}
\overleftarrow{\Theta}_0:=\{P\in \overleftarrow{\Theta}|& & E_{P}[\phi_n(\xi)\psi_m(\theta_n\xi)]\le \overleftarrow{\mathbb{E}}[\phi_n(\xi)] E_P[\psi_m(\theta_n\xi)],  \ \textmd{for any}\\
& & \ \phi_n\in C_{b,Lip}(\mathbb{R}^n), \ \psi_m\in C^{+}_{b, Lip}(\mathbb{R}^m),  \ m, n\in \mathbb{N}\}.
\end {eqnarray*}
In the above, we abbreviate $\phi_n(\xi_1, \cdots, \xi_n)$ as $\phi_n(\xi)$ for $\phi_n\in C_{b, Lip}(\mathbb{R}^n)$. 

\begin {lemma} \label {Char-P-left}
$\overleftarrow{\Theta}=\overleftarrow{\Theta}_0$.
\end {lemma}
\begin {proof}
We first give some properties of $\overleftarrow{\Theta}_0$.

1)  $\overleftarrow{\Theta}_0$ is convex.

For $P_1, P_2\in \overleftarrow{\Theta}_0 $, and $\alpha\in (0,1)$, clearly we have $Q:=\alpha P_1+(1-\alpha) P_2\in \overleftarrow{\Theta}$.  For any  $m, n\in \mathbb{N}$ and $\phi_n\in C_{b,Lip}(\mathbb{R}^n), \ \psi_m\in C^{+}_{b, Lip}(\mathbb{R}^m)$,
\begin {eqnarray*}
& &E_Q[\phi_n(\xi)\psi_m(\theta_n\xi)]\\
&=&\alpha E_{P_1}[\phi_n(\xi)\psi_m(\theta_n\xi)]+(1-\alpha)E_{P_2}[\phi_n(\xi)\psi_m(\theta_n\xi)]\\
&\le& \alpha \overleftarrow{\mathbb{E}}[\phi_n(\xi)]E_{P_1}[\psi_m(\theta_n\xi)] +(1-\alpha) \overleftarrow{\mathbb{E}}[\phi_n(\xi)]E_{P_2}[\psi_m(\theta_n\xi)]\\
&=& \overleftarrow{\mathbb{E}}[\phi_n(\xi)]E_{Q}[\psi_m(\theta_n\xi)].
\end {eqnarray*}

2) $\overleftarrow{\Theta}_0$ is weakly compact.

Let $(P_k)\subset \overleftarrow{\Theta}_0$ converge weakly to $P$. Then $P$ belongs to $\overleftarrow{\Theta}$. For any  $m, n\in \mathbb{N}$ and $\phi_n\in C_{b,Lip}(\mathbb{R}^n), \ \psi_m\in C^{+}_{b, Lip}(\mathbb{R}^m)$,
\begin {eqnarray*}
& &E_P[\phi_n(\xi)\psi_m(\theta_n\xi)]\\
&=&\lim_{k\rightarrow\infty}E_{P_k}[\phi_n(\xi)\psi_m(\theta_n\xi)]\\
&\le&\lim_{k\rightarrow\infty} \overleftarrow{\mathbb{E}}[\phi_n(\xi)]E_{P_k}[\psi_m(\theta_n\xi)] \\
&=& \overleftarrow{\mathbb{E}}[\phi_n(\xi)] E_{P}[\psi_m(\theta_n\xi)].
\end {eqnarray*}

3) $\sup\limits_{P\in\overleftarrow{\Theta}_0}E_{P}[X]=\overleftarrow{\mathbb{E}}[X]$, for $X\in C_b(\Omega)$.

Set $Lip(\Omega)=\{\varphi(\xi_1, \cdots, \xi_n)| \ \varphi\in C_{b, Lip}(\mathbb{R}^n), \ n\in \mathbb{N}\}$. For $X\in Lip(\Omega)$, the assertion holds by the assumption that $\xi_1, \xi_2, \cdots$ is $<$i.i.d under $\overleftarrow{\mathbb{E}}$. Noting that $Lip(\Omega)$ is dense in $C_b(\Omega)$ with respect to the norm $\overleftarrow{\mathbb{E}}[|\cdot|]$, we get the desired result.

Then, it follows from Hahn-Banach theorem that $\overleftarrow{\Theta}=\overleftarrow{\Theta}_0$.
\end {proof}

The original manuscript of this paper with the above result was completed in September, 2019. Recently, we found a similar result in \cite {HLL21}.
\begin {theorem} \label {Triviality-left} Let $\overleftarrow{\Theta}^e$ be the subset of the extreme points of the convex set $\overleftarrow{\Theta}$. Then, for any $A\in\mathcal{T}$ and $P\in\overleftarrow{\Theta}^e$, we have
\[\textmd{either} \ P(A)=0, \ \textit{or} \ P(A)=1.\]
\end {theorem}
\begin {proof} Let $P\in\overleftarrow{\Theta}$. For $m, n\in\mathbb{N}$, $m> n$, and $\phi\in C_{b,Lip}(\mathbb{R}^n)$, it follows from Lemma \ref {Char-P-left} that $E_P[\phi(\xi_1, \cdots, \xi_n)| \mathcal{G}_m]\le \overleftarrow{\mathbb{E}}[\phi(\xi_1, \cdots, \xi_n)]$, $P$-a.s. By the martingale convergence theorem, we have $E_P[\phi(\xi_1, \cdots, \xi_n)| \mathcal{G}_m]\rightarrow E_P[\phi(\xi_1, \cdots, \xi_n)| \mathcal{T}]$, $P$-a.s. So we have
\[E_P[\phi(\xi_1, \cdots, \xi_n)| \mathcal{T}]\le \overleftarrow{\mathbb{E}}[\phi(\xi_1, \cdots, \xi_n)],\]
and, consequently, for $X\in C_b(\Omega)$, $E_P[X| \mathcal{T}]\le \overleftarrow{\mathbb{E}}[X].$ Then we can choose a regular conditional probability $P_{\omega}$  of $P[\cdot |\mathcal{T}]$ such that $P^{\omega}\in \Theta$ for each $\omega\in \Omega$. So, the following two facts hold.
\begin{itemize}
  \item[1)] $P_{\omega}(A)=1_{A}(\omega)$, $P$-a.s., for $A\in\mathcal{T}$.
  \item[2)] If $P\in\overleftarrow{\Theta}^e$, $P_{\omega}=P$, $P$-a.s.
\end{itemize}
Then,  for $A\in\mathcal{T}$ and $P\in\overleftarrow{\Theta}^e$, $P(A)=1_{A}(\omega)$, $P$-a.s., which implies that $P(A)=0, \ or \ 1.$
\end {proof}
\begin {corollary}Let $\overleftarrow{\Theta}^e$ be the subset of the extreme points of the convex set $\overleftarrow{\Theta}$. Then, for any $P\in\overleftarrow{\Theta}^e$, $\xi_{\infty}^{\triangledown}$ (resp. $\xi_{\infty}^{\vartriangle}$) is equal to a constant, $P$-a.s.
\end {corollary}

\subsection {Tail $\sigma$-algebra of i.i.d random variables} 

In this subsection, we investigate the triviality of the tail $\sigma$-algebra of i.i.d random variables under a regular sublinear expectation $\mathbb{E}$ defined on $\Omega=\mathbb{R}^{\mathbb{N}}$.

Although it does not hold under $P\in\Theta^e$ generally, we shall show that the tail $\sigma$-algebra of i.i.d random variables is trivial under 
each ``stationary" probability $P\in \Theta^e$, denoted by $\Theta^s$, and that
 \[\sup_{P\in \Theta^s}E_P[X]=\sup_{P\in \Theta}E_P[X]=\mathbb{E}[X], \ \textmd{for} \ X\in C_b(\Omega).\]
 
For $n\in\mathbb{N}$, let $\Theta_n$ be the subset of $\Theta$ such that for $P\in\Theta_n$ one has $P\circ\theta_n^{-1}=P$. 
Set $\Theta_{\infty}=\cup_n\Theta_n$.  
\begin {theorem}\label {Triviality-right} Let $\Theta^s=\Theta^e\cap\Theta_{\infty}$. 
\begin{itemize}
  \item[1)] For any $P\in\Theta^s$ and $A\in\mathcal{T}$, we have \[\textmd{either} \ P(A)=0, \ \textmd{or} \ P(A)=1;\]
  \item[2)] For any $P\in\Theta^s$,  we have $\xi_{\infty}^{\triangledown}=\xi_{\infty}^{\vartriangle}=m_P$, $P$-a.s., for some constant $m_P$; Moreover, $\{m_P\}_{P\in\Theta^s}$ is dense in $\big[\underline{\mu}, \ \overline{\mu}\big]$.
  \item[3)] $\Theta^s$ is a set that represents $\mathbb{E}$: \[\sup_{P\in \Theta^s}E_P[X]=\mathbb{E}[X], \ \textit{for} \ X\in C_b(\Omega).\]
  \end{itemize}
\end {theorem}
To prove the above theorem, we first give a characterization of the set $\Theta$, which is a counterpart of Lemma \ref{Char-P-left}, and the proof is similar.
\begin {lemma}\label {Char-P-right} Let $P\in\Theta$. For $n, m\in\mathbb{N}$, $\phi_n\in C_b^+(\mathbb{R}^n)$ and $\psi_m\in C_b(\mathbb{R}^m)$, we have
\begin {eqnarray}\label {Theta-right}E_P[\phi_n(\xi)\psi_m(\theta_n\xi)]\le E_P[\phi_n(\xi)]\mathbb{E}[\psi_m(\theta_n\xi)].
\end {eqnarray}
Here we write $\phi_n(\xi)$ for $\phi_n(\xi_1, \cdots, \xi_n)$.
\end {lemma}

An equivalent statement of (\ref{Theta-right}):  Let $P_{\omega}$ be a regular conditional probability of $P[\cdot|\mathcal{F}_n]$. Then 
\begin {eqnarray}\label {Theta-right2}
P_{\omega}\circ\theta_n^{-1}\in \Theta, \ P\textmd{-a.s.}
\end {eqnarray}

\begin {proof}[Proof to Theorem \ref {Triviality-right}] For $n\in\mathbb{N}$, set $\Xi_n=\{P_{1,n}=P\circ(\xi_1,\cdots, \xi_n)^{-1}|P\in\Theta\}$, which is convex and closed. Let $\Xi^e_n$ be the set of extreme points of $\Xi_n$, and set \[\Theta^*_n=\{\otimes_{k=1}^{\infty}\mu_k| \ \mu_k=\mu\in\Xi^e_n, \ k\in\mathbb{N}\}.\]
{\bf Claim:}  $\Theta^e\cap\Theta_{n}=\Theta^*_n$, for any $n\in\mathbb{N}$.

Clearly, the conclusions 1), 3) of the theorem follow immediately from this claim. For the conclusion 2), notice that, for $P\in\Theta^*_n$, we have
$\xi_{\infty}^{\triangledown}=\xi_{\infty}^{\vartriangle}=\frac{1}{n}E_P[\Sigma_{k=1}^n\xi_k]=:m_P$, $P$-a.s. By arguments similar to the proof to Theorem \ref {LLN-P-kappa},
we get the $\{m_P\}_{P\in\Theta^s}$ is dense in $\big[\underline{\mu}, \ \overline{\mu}\big]$.

 Let us prove the claim now.

a) $\Theta^e\cap\Theta_{n}\subset\Theta^*_n$.

We first prove that $\theta_n\xi$ is independent of $\xi_1, \cdots, \xi_n$ under $P\in\Theta^e\cap\Theta_{n}$. In fact, for any $\phi\in C_b(\mathbb{R}^m), \ m\in\mathbb{N}$, we have
\begin {eqnarray*}
E_{P}[\phi(\xi)]&=&E_{P}[\phi(\theta_{n}\xi)]\\
&=&E_P[E_{P}[\phi(\theta_{n}\xi)|\mathcal{F}_n]]\\
&=&E_P[E_{P_{\omega}\circ\theta_n^{-1}}[\phi(\xi)]],
\end {eqnarray*} where $P_{\omega}$ is a regular conditional probability of $P(\cdot|\mathcal{F}_n)$.
Since $P_{\omega}\circ\theta_n^{-1}\in\Theta$ by Lemma \ref {Char-P-right} and $P\in\Theta^e$, we conclude that $E_{P}[\phi(\theta_{n}\xi)|\mathcal{F}_n]$ is a constant.

Now we show that $P_{1,n}\in\Xi_n^e$. If $P_{1,n}\notin\Xi_n^e$, there exist $\alpha\in(0,1)$, $\mu,\nu\in\Xi_n$, $\mu\neq\nu$ such that $\alpha\mu+(1-\alpha)\nu=P_{1,n}$. Set $Q=\mu\otimes P$ and $R=\nu\otimes P$. Then $Q,R\in\Theta$, $Q\ne R$ and $P=\alpha Q+(1-\alpha)R$. This is a contradiction.

b) $\Theta^e\cap\Theta_{n}\supset\Theta^*_n$.

Let $P\in\Theta^*_n$. We only need to prove that $P\in\Theta^e$. If there exit $\alpha\in(0,1)$, $Q, R\in\Theta$ such that $P=\alpha Q+(1-\alpha)R$, then $P_{i,j}=\alpha Q_{i,j}+(1-\alpha)R_{i,j}$ for any $i,j\in\mathbb{N}$, $i\le j$. Since $P_{in+1,(i+1)n}\in\Xi^e_n$ and $Q_{in+1,(i+1)n}, R_{in+1,(i+1)n}\in\Xi_n$, we have 
\[P_{in+1,(i+1)n}=Q_{in+1,(i+1)n}=R_{in+1,(i+1)n}, \  for \  i\ge0.\] 
Assuming that $P_{1, kn}=Q_{1,kn}=R_{1,kn}$ for some $k\in\mathbb{N}$, we shall prove that 
\[P_{1, (k+1)n}=Q_{1,(k+1)n}=R_{1,(k+1)n}.\] 
Let $Q_{\omega}$ be a regular conditional probability of $Q[\cdot|\mathcal {F}_{kn}]$. By Lemma \ref {Char-P-right}, we have $Q_{\omega}\circ\theta_{kn}^{-1}\in\Theta$, $Q$-a.s., and consequently, $Q_{\omega}\circ\theta_{kn}^{-1}[(\xi_1,\cdots, \xi_n)\in\cdot]\in\Xi_n$, $Q$-a.s. Note that, for $\phi\in C_b(\mathbb{R}^n)$,
\[E_{Q_{kn+1,(k+1)n}}[\phi]=E_Q[\phi(\theta_{kn}\xi)]=E_Q[E_Q[\phi(\theta_{kn}\xi)|\mathcal{F}_{kn}]]=E_Q[E_{Q_{\omega}\circ\theta_{kn}^{-1}}[\phi(\xi)]].\]
Since  $Q_{kn+1,(k+1)n}=P_{kn+1,(k+1)n}\in\Xi^e_n$, we have 
\[E_{Q}[\phi(\theta_{kn}\xi)|\mathcal {F}_{kn}]=E_{Q_{\omega}\circ\theta_{kn}^{-1}}[\phi(\xi)]=E_{P_{kn+1,(k+1)n}}[\phi], \ P_{1,kn}\textmd{-a.s.}\] 
So we have $P_{1, (k+1)n}=Q_{1,(k+1)n}$ and, similarily, $P_{1, (k+1)n}=R_{1,(k+1)n}$. \end {proof}

\section {Some examples}
A typical sublinear expectation that satisfies the assumptions imposed in this paper is $G$-expectation space introduced by Peng S.  In this section, we present some other examples.
\begin {example}\label {YY}
Let $\Omega=\{0,1\}^{\mathbb{N}}$. For $\omega, \bar{\omega}\in\Omega$, define $d(\omega, \bar{\omega})=\Sigma_{k=1}^{\infty}\frac{1}{2^k}|\omega(k)-\bar{\omega}(k)|$. For $\omega\in\Omega$, let $\delta_{\omega}$ be the Dirac measure on $\Omega$. Set $\Theta=\{\delta_{\omega}| \ \omega\in\Omega\}$ and $\xi_n(\omega)=\omega(n)$.
It is easy to check that $\Omega$ is compact with respect to the metric $d$, and that $\xi_n$ is continuous on $(\Omega, d)$ for any $n\in\mathbb{N}$. We claim that
 $\{\xi_n\}$ is $i.i.d.$  under the sublinear expectation $\mathbb{E}=\sup_{\omega\in\Omega}E_{\delta_\omega}$. 
 
 It suffices to prove that $\xi_m$ is independent of $\xi_1, \cdots, \xi_{m-1}$,  for any $m\in\mathbb{R}$. In fact,  for any $\phi\in C_{b, Lip}(\mathbb{R}^m)$, 
 and $(x_1, \cdots, x_{m-1})\in\{0,1\}^{m-1}$, we have 
 \[\mathbb{E}[\phi(x_1, \cdots,x_{m-1}, \xi_m)]=\max\limits_{y\in\{0,1\}}\phi(x_1, \cdots, x_{m-1},  y)\le\max\limits_{\vec{y}\in\{0,1\}^m}\phi(y_1, \cdots, y_{m})=\mathbb{E}[\phi(\xi_1, \cdots, \xi_m)].\] Therefore,
  \[\mathbb{E}[\mathbb{E}[\phi(x_1, \cdots,x_{m-1}, \xi_m)]|_{x_i=\xi_i}]\le\mathbb{E}[\phi(\xi_1, \cdots, \xi_m)].\]
 On the other hand, assume $\mathbb{E}[\phi(\xi_1, \cdots, \xi_m)]=\phi(x^0_1, \cdots, x^0_m)$ for some $(x^0_1, \cdots, x^0_m)\in\{0,1\}^m$. Then, we have
 $\mathbb{E}[\phi(x^0_1, \cdots, x^0_{m-1}, \xi_m)]=\phi(x^0_1, \cdots, x^0_m),$ and, consequently,
  \[\mathbb{E}[\mathbb{E}[\phi(x_1, \cdots,x_{m-1}, \xi_m)]|_{x_i=\xi_i}]\ge\mathbb{E}[\phi(\xi_1, \cdots, \xi_m)].\] 
  It follows from the above arguments that $\xi_m$ is independent of $\xi_1, \cdots, \xi_{m-1}$.
   
  For this example, it is easily seen that the conclusions in Theorem \ref {LLN-bound}, \ref {LLN-P-kappa} hold.
  
\end {example}

\begin {example}  Let $\Omega$ be the same space as is defined in Example \ref {YY}, and let $\Omega_0$ be the subset of $\Omega$ with $\xi_i=0$ for all except a finite number $i\in\mathbb{N}$.

i) (\cite {Teran18}) Let $\Theta_0=\{\delta_{\omega}| \omega\in\Omega_0\}$ and  $\mathbb{E}^0=\sup_{\omega\in\Omega_0}E_{\delta_\omega}$. 

ii) Set $\hat{\xi}_n=\xi_n1_{\Omega_0}$ and  $\mathbb{E}=\sup_{\omega\in\Omega}E_{\delta_\omega}$.

Similar to Example \ref {YY}, it can be shown that $\{\xi_n\}$ (resp. $\{\hat{\xi}_n\}$) is $i.i.d.$  under the sublinear expectation $\mathbb{E}^0$ (resp. $\mathbb{E}$). 

For these two examples, it is easily seen that the conclusion in Theorem \ref {LLN-bound} still holds, but Theorem \ref {LLN-P-kappa}  fails since the assumptions there are violated. More precisely, $\Theta_0$ in i) is not weakly compact, and $\hat{\xi}_n$ in ii) is not (quasi-)continuous.

\end {example}

\section*{Acknowledgements}
The author would like to thank Professor Shige Peng for very fruitful discussions. The author is financially supported by National Key R\&D Program of China (No. 2020YFA0712700 \& No. 2018YFA0703901), NSFCs (No. 11871458 \& No. 11688101); and
Key Research Program of Frontier Sciences, CAS (No. QYZDB-SSW-SYS017).

%%% ----------------------------------------------------------------------

%%%%%%%%%%%%%%%%%%%%%%%âÃ¢Â§âÃ­ââÂ¬â«âÃ­âÃ­âÃ¬âÃ³Ã€
\renewcommand{\refname}{\large References}{\normalsize \ }

\end{document}